\documentclass[12pt]{article}%
\usepackage{amsthm}
\usepackage{amsmath}
\usepackage{graphicx}%
\usepackage{amsfonts}%
\usepackage{amssymb}
\usepackage{tikz}
\providecommand{\U}[1]{\protect\rule{.1in}{.1in}}
\providecommand{\U}[1]{\protect\rule{.1in}{.1in}}
\newtheorem{lemma}{Lemma}

\newcommand{\C}{\mathbb C}

\newcommand{\R}{\mathbb R}
\newcommand{\Z}{\mathbb Z}

\DeclareMathOperator{\area}{area}
\DeclareMathOperator{\diag}{diag}

\DeclareMathOperator{\GL}{GL}
\DeclareMathOperator{\SL}{SL}
\DeclareMathOperator{\SO}{SO}
\DeclareMathOperator{\sgn}{sgn}

\DeclareMathOperator{\vol}{vol}

\begin{document}
\title{About the value of the two dimensional Levy's constant}
\author{Yitwah Cheung$^1$ \and Nicolas Chevallier$^2$}
\date{%
    $^1$Yau Mathematical Sciences Center, Tsinghua University\\%
    $^2$Département de Mathématiques, Université de Haute Alsace\\%
    \bigskip
    \today
}
\maketitle

\begin{abstract}
    We give a numerical approximation of the Lévy constant on the growth of the denominators of the best Diophantine approximations in dimension 2 with respect to the euclidean norm.
This constant is expressed as an integral on a surface of dimension 7. We reduce the computation of this integral to a triple integral, whose numerical evaluation was carried out in \cite{Xieu}.  
\end{abstract}

\section{Introduction}

In 1936, Aleksandr Khintchin showed that there exists a constant $K$ such that the denominators $(q_{n})_{n\geq 0}$
of the convergents of the continued fraction expansions of almost all real
numbers $\theta$ satisfy
\[
\lim_{n\rightarrow\infty}\frac1n\ln q_{n}=K
\]
 Soon afterward, Paul L\'{e}vy gave the explicit value
of the constant,
\[
K=\frac{\pi^{2}}{12\ln 2}.
\]
In \cite{CheChe}, this result is extended to the denominators of best Diophantine approximations to vectors in $\R^d$ and even to matrices in $M_{d,c}(\R)$. The value of the limit is given by an integral $\int_Sd\mu_S$ over a codimension one submanifold $S$ in the space of lattices $\SL(d+1,\R)/\SL(d+1,\Z)$ (see Section~\ref{S:main:formula} below). However, apart from the case of $d=1$, this integral is very difficult to calculate. The aim of this document is to give a numerical 
approximation of the integral associated with best Diophantine approximations to  vectors in $\R^2$.  

This document is organized as follows. We first give the definition of the submanifold $S$ together with a parametrization of $S$. Then we give an explicit formula for the measure $\mu_S$ induced by the flow. The two difficult parts of the work are the explicit description of the domain of integration, i.e., the subset of parameters corresponding to $S$, and the calculation of the integral $\int_Sd\mu_S$. This is done in the two last sections.   
For more details on best Diophantine approximation we refer the reader to \cite{CheChe}.

\section{Definitions of $S$ and its parametrization}

\subsection{Definition}
The surface $S$ is the set of unimodular lattices $\Lambda$ in $\R^3$ such
that there exist two independent vectors $u=(u_1,u_2,u_3)$ and $v=(v_1,v_2,v_3)$ in
$\Lambda$ such that:

\begin{itemize}
\item $|u_3|$ and $\sqrt{v_1^2+v_2^2}$ are $<r=|v_3|=\sqrt{u_1^2+u_2^2}$,

\item the only nonzero points of $\Lambda$ in the cylinder $C(r)=\{(x_1,x_2,x_3)\in\R^3:\max(\sqrt{x_1^2+x_2^2},|x_3|)\leq r\}$ are $\pm u$ and $\pm v$.
\end{itemize}

Let $\mu_G$ be the Haar measure in $G=\SL(3,\R)$. Let $\mu$ be the measure in the space  of unimodular lattices $\SL(3,\R)/\SL(3,\Z)$ invariant by the left action of $G$, induced by $\mu_G$. In turn, let $\mu_S$ be the measure induced by $\mu$ and the diagonal flow
\[
g_t=\operatorname{diag}(e^t,e^t,e^{-2t}),\, t\in\R
\]
on $S$.

\subsection{The main formula}\label{S:main:formula}
In \cite{CheChe}, it is proved that for Lebesgue-almost all $\theta\in\R^2$,
\[
\lim_{n\rightarrow\infty}\frac1n\ln q_n(\theta)=\frac{2\mu(\SL(3,\R)/\SL(3,\Z))}{\mu_S(S)}
\]
where $(q_n(\theta))_n$ is the sequence of best approximation denominators of $\theta$ associated with the standard Euclidean norm in $\R^2$.

\subsection{Parametrization of $S$}
Let $\Lambda$ be a lattice in $S$ and let $u=(u_1,u_2,u_3)$ and $v=(v_1,v_2,v_3)$ be the two vectors associated with $\Lambda$ by the definition of $S$. By Lemma 10 of \cite{CheChe}, there exists a vector $w\in \Lambda$ such that $u,v,w$ is a basis of $\Lambda$. We can suppose $u_3$ and $v_3\geq 0$ w.l.g.. 
There is a rotation $k_{\theta}$ in $\SO(3,\R)$ that fixes $e_3$ and such that $rk_{\theta}e_1=(u_1,u_2,0)$ where $r=\sqrt{u_1^2+u_2^2}$. If $M$ is the $3\times 3$ matrix whose columns are $u,v$ and $w$, we have
\begin{align*}
    \Lambda&=M\Z^3,\\
    M&=rk_{\theta}\begin{pmatrix}
    1&a_1&c_1\\
    0&a_2&c_2\\
    b&1&c_3
    \end{pmatrix}
\end{align*}
where $b\in[0,1[$, $a_1^2+a_2^2<1$ and 
\[
\det M=r^3((1-a_1b)(-c_2) - a_2(bc_1-c_3))=1.
\]
Therefore, we can parametrize $S$ with the seven parameters
\[
\theta,b,a_1,a_2,c_1,c_2,c_3.
\]
The problem is now to find a subset of parameters $\Omega_7$ such that
\begin{itemize}
    \item $\pm u$ and $\pm v$ are the only nonzero vector of $\Lambda$ in the cylinder $C(r)$,
    \item for every $\Lambda\in S$ there exists exactly one $7$-tuple $(\theta,b,a_1,a_2,c_1,c_2,c_3)$ such that $\Lambda=M\Z^3$,
\end{itemize} see the section about the Domain of integration.

\section{Induced measure on $S$}
We use Siegel normalization of the Haar measure on $G=\SL(3,\R)$, see \cite{Siegel} Lecture XV. For a Borel set $B\subset \SL(3,\R)$,
\[
\mu_G(B)=\operatorname{Lebesgue}_{\R^9}(B')
\]
where $B'=\{tM:M\in B,\,t\in[0,1]\}$.
Consider the parametrization of $\GL_+(3,\R)$ given by 
  $$M=g_trk_{\theta}\begin{pmatrix}
    1&a_1&c_1\\
    0&a_2&c_2\\
    b&1&c_3
    \end{pmatrix}.$$  
In these coordinates the standard volume form in $\R^9$ 
is\footnote{This was obtained by computing the determinant of a 9-by-9 matrix.  The factor $r^8$ is expected as $r$ is homogeneous of degree one. For $g_t=\exp(\diag(\lambda_1t,\lambda_2t,\lambda_3t))$ the leading coefficient generalizes to $\lambda_1-\lambda_3.$}
  $$3r^8d(r,t,\theta,a_1,a_2,b,c_1,c_2,c_3).$$

We wish to replace $r$ with the homogeneous coordinate $\rho:=\Delta^{1/3}$ where 
    $$\Delta =\operatorname{det}M= r^3\eta, \quad\text{ and }\quad 
    \eta := (1-a_1b)(-c_2) - a_2(bc_1-c_3)>0.$$ 
Rewrite Lebesgue volume form in the new coordinates as 
  $$\frac{\Delta^2d(\Delta,t,\theta,a_1,a_2,b,c_1,c_2,c_3)}
  {\eta^3} = 
  \frac{3\rho^8d(\rho,t,\theta,a_1,a_2,b,c_1,c_2,c_3)}
  {\eta^3}.$$
The local form of Haar measure with Siegel's normalization is
  $$d\mu_G=\int_0^1\frac{3\rho^8d(t,\theta,a_1,a_2,b,c_1,c_2,c_3)}
  {\eta^3}d\rho=\frac{d(t,\theta,a_1,a_2,b,c_1,c_2,c_3)}
             {3((1-a_1b)(-c_2) - a_2(bc_1-c_3))^3}.$$
Using this normalization, we can quote Siegel's formula and bypass the computation of the first return time:
  $$\vol(\SL(3,\R)/\SL(3,\Z)) = \frac{\zeta(2)\zeta(3)}{3}.$$  
The local form for the induced measure on the transversal $S$ is now 
    $$d\mu_S = \frac{d(\theta,a_1,a_2,b,c_1,c_2,c_3)}
             {3((1-a_1b)(-c_2) - a_2(bc_1-c_3))^3}.$$
Levy's constant is $d=2$ times the average return time (see the main formula), or 
  $$K= L_{2,1}=\frac{2\zeta(2)\zeta(3)}{3\mu_S(\Omega_7)}$$  
where $\Omega_7$ is any domain of integration parametrizing the transversal $S$.  

\section{Domain of integration}
The suspension of the transversal gives a fundamental domain for the action of $\SL(3,\Z)$ that is invariant under is $\SO(2)$ rather than $\SO(3)$.  The quotient of the 
transversal by the circle action is a 6 dimensional space that we shall realize as a fiber bundle over the 3 dimensional base $\Omega_2\times[0,1[$ where $$\Omega_2 = \{a=a_1+ia_2\in\C: |a|<1, |a-1|\ge1\}$$ 
parametrizes the possible configurations of $v=(a_1,a_2,1)$ and $u=(1,0,b)$ on the boundary of the standard unit cylinder $C_0=C(1)$.  Given a lattice in our transversal $S$ we rescale and rotate so that the systole cylinder is given by $C_0$ with the vectors $u$ and $v$ on its $\partial_+$- and $\partial_-$-faces.  Ambiguity in the normal form can be ignored since it occurs on a set of positive codimension.  The lattice is determined by specifying the third vector $w=(c_1,c_2,c_3)$ that forms an integral basis for the lattice, positively oriented by requiring $c_2<0$.  We have some freedom for the choice of $w$, which we will describe next.  Once this choice is made, the projection from the 6 dimensional total space to the base will have been specified.

The space of possibilities for $w$ is a connected component of the complement of the union of $C_0$ with all its translates under the action of the discrete subgroup $\Z u+\Z v$.  The topological boundary is tesselated by the domain $G(a,b)$ where this surface meets $\partial C_0$.  The orthogonal projection of $G(a,b)$ onto the $(c_1,c_3)$-plane is a rectilinear domain that we will denote by $$F(a,b).$$
The choice of $w$ is such that its orthogonal projection lies in $F(a,b)$;
that is, $$(c_1,c_3)\in F(a,b) \quad\text{ and }\quad -c_2\ge\sqrt{1-c_1^2}.$$
$F(a,b)$ is rectilinear since curved arcs on $\partial G(a,b)$ map to horizontal segments.

\subsection{Intersection patterns}
For the precise description of $F(a,b)$ we need to know which cylinders appear on $\partial G(a,b)$.  
By symmetry, it suffices to consider the case $a_2>0$.
\begin{lemma}\label{lem:cylinder_pairs}
$C_{\pm u}, C_{\pm v}, C_{\pm(u+v)}$ and $C_{\pm(v-u)}$ are the only translates whose intersection 
with $C_0$ has nonempty interior, unless $|2a-1|<2$, when 
there is an additional pair $C_{\pm(2v-u)}$.
\end{lemma}
\begin{proof}
The hypothesis on $C_{mu+nv}$ is satisfied by $(m,n)$ in the intersection of the infinite strip $|bm+n|<2$ and an ellipse that is contained in the vertical strip $|m|<2|a_1|/|a|$.  
The pair $(2,-1)$ lies in the ellipse if and only if $|2a-1|<2$.
\end{proof}

\begin{lemma}\label{lem:C(2v-u)}
$C_{\pm(2v-u)}$ is disjoint from $\partial C_0\setminus U$ 
where $U$ is the union of the cylinders in Lemma~\ref{lem:cylinder_pairs}.
\end{lemma}
\begin{proof}
By symmetry, it suffices to show $C_{2v-u}\cap\partial C_0\subset \mathrm{Int~} C_{v-u}\cup C_v$, which in terms of euclidean disks is the condition $D_0\cap D_{2a-1} \subset \mathrm{Int~} D_a\cup D_{a-1}$.
For this, it is enough to verify that $D_0\cap D_{2a-1}$ is contained in the open disk of radius 1/2 centered at $a-1/2$, i.e. 
  $$ \left| \chi_-(2a-1) - a + \frac12 \right| < \frac12 $$
  which is equivalent to $|2a-1|<2$.  Here, the notation 
  $$\chi_\pm(a) := \frac{a}{2}\pm\frac{ia}{|a|}\sqrt{1-\frac{|a|^2}{4}}$$
 denotes the two points where the circles $|z|=1$ and $|z-a|=1$ meet.  
\end{proof}

To show that a cylinder meets $\partial G(a,b)$ in an arc involves 
checking that the arc is disjoint from the other 7 cylinders.  
There does not seem to be an easier way to carry out this tedious task, 
e.g. to verify that the cylinders $C_{\pm u}$ and $C_{\pm v}$ appear 
on $\partial G(a,b)$ in every possible scenario.\footnote{Later, we shall show that the projection $F(a,b)$ forms a fundamental domain for the action of $\Z (1,b)+\Z (-a_1,1)$, which indirectly shows that none of the arcs on $\partial F(a,b)$ can be blocked by the other cylinders.} 
We shall skip the verification of these intuitive claims and simply 
identify the criteria for the appearance of the other 4 cylinders.  It is perhaps surprising that the answer is always determined by the sign of $|a-\xi|-1$ where $\xi$ denotes the sixth root of unity in the first quadrant.
\begin{lemma}\label{lem:Omega:I}
$C_{\pm(u-v)}$ appear iff $|a-\xi|<1$ while 
$C_{\pm(u+v)}$ appear iff $|a-\xi|>1$.  
\end{lemma}
\begin{proof}
Considerations of elementary nature lead to the following characterizations: 
$C_{u-v}$ appears iff $d(1-a,\Bar\xi)<1$ while 
$C_{v-u}$ appears iff $d(a-1,-\xi)<1$; 
$C_{u+v}$ appears iff $d(a+\Bar\xi,1)>1$ while 
$C_{-u-v}$ appears iff $d(-a-\Bar\xi,-1)>1$.
In each case the sign of $|a-\xi|-1$ is nonzero.
\end{proof}

The shape of $F(a,b)$ also depends on how the cylinders intersect with $C_0$.
\begin{lemma}\label{lem:Omega:III}
$G(a,b)$ has overlap with the bottom of $C_0$ iff $|a+\xi|<1$.
\end{lemma}
\begin{proof}
Note that the point where the curved faces of $C_{-v}$ and $C_{-u-v}$ meet the plane $z=-1$ is described by the complex number $-a-\xi$.  
\end{proof}

\begin{figure}
\begin{center}
\begin{tikzpicture}[scale=4]
\draw (-1,0) -- (0,0);
\draw[dashed] (0,0) -- (1,0) node[below] {$1$};
\draw (.5,.866) arc (60:180:1);
\draw[dashed] (1,0) arc (0:60:1) node[above right] {$\xi$};
\draw (.5,.866) arc (120:180:1);
\draw (-.5,.866) node[above left] {$-\Bar{\xi}$} arc (180:240:1);
\draw (0,0) arc (60:120:1);
\draw (0,.5) node [above left] {$\Omega_{I}^+$};
\draw (-.5,.3) node[above left] {$\Omega_{II}^+$};
\draw (-.5,0) node [below] {$\Omega_{III}^+$};
\end{tikzpicture}
\end{center}
\caption{Three subregions of $\Omega_2^+$.}  
\label{fig:Omega_+}
\end{figure}
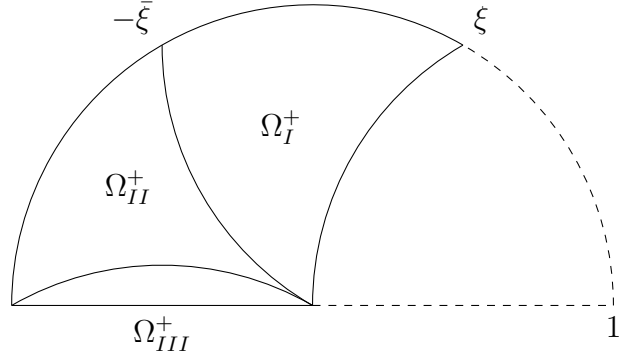

Let us divide $\Omega^+_2:=\{ a\in\Omega_2 : a_2\ge0 \}$ into the following subregions as depicted in Figure~\ref{fig:Omega_+}:
\begin{align*}
  \Omega_I^+ &= \{ a\in\Omega_2 : |a-\xi|<1\} \\
  \Omega_{II}^+ &= \{ a\in\Omega_2 : |a-\xi|\ge1, |a+\xi|\ge1 \} \\
  \Omega_{III}^+ &= \{ a\in\Omega_2 : |a+\xi|<1, \text{Im~}a\ge0 \} 
\end{align*}

Next, we describe $G(a,b)$ in each of the 3 main cases.  

CASE $a\in\Omega^+_I$:
$G(a,b)$ is bounded by the cylinders $C_{\pm u}$, $C_{\pm v}$ and  $C_{\pm(v-u)}$ and consists of two subregions $G_+\sqcup G_0$ joined along an arc, labeled $\oplus$, along the top rim of $C_0$.  The arcs along the boundary of each subregion in counter-clockwise order are labeled 
\begin{align*}
  G_+&: \oplus, u, v, v-u \\
  G_0\,&: \oplus, v-u, -u, -u, -v, u-v, u-v, u 
\end{align*}

CASE $a\in\mathrm{Int~}\Omega^+_{II}$:
$G(a,b)$ is bounded by the cylinders $C_{\pm u}$, $C_{\pm v}$ and $C_{\pm(v+u)}$ and consists of two subregions as in the previous case, with arcs along the boundary of the subregions labeled 
\begin{align*}
  G_+&: \oplus, u, v+u, v \\
  G_0\,&: \oplus, v, -u, -u, -v-u, -v, -v, u 
\end{align*}

CASE $a\in\Omega^+_{III}$:
$G(a,b)$ is bounded by the same cylinders in the previous case 
but now consists of three subregions $G_+\sqcup G_0\sqcup G_-$ bounded by the following arcs
\begin{align*}
  G_+&: \oplus, v+u, v \\
  G_0\,&: \oplus, v, v, -u, -v-u, -v-u, \ominus, -v, -v , u, v+u, v+u \\
  G_-&: \ominus, -v-u, -v
\end{align*}
where $\ominus$ is an arc along the bottom rim of $C_0$ joining $G_0$ and $G_-$.  

In the next section, we use the above description of $G(a,b)$ to arrive at an explicitly described region $\Tilde F(a,b)$ that is \emph{a priori} only known to contain $F(a,b)$.  We will verify that $\Tilde F(a,b)$ is a fundamental domain for the action of $\Z(1,b)+\Z(-a_1,1)$, from which it follows that $\Tilde F(a,b)$ provides an equivalent definition of $F(a,b)$.  In anticipation of this conclusion and since there is no further need to distinguish between to two sets, we shall drop the overscript when referring to $\Tilde F(a,b)$ in the following section.  

\subsection{Explicit description of $F(a,b)$}
Each vertical arc on $\partial F(a,b)$ is the projection (in the $y$-direction) of a linear segment on $\partial G(a,b)$ along which one of the cylinders in Lemma~\ref{lem:cylinder_pairs} is transverse to $C_0$.  Using complex notation for the projection (in the $z$-direction) of the linear segment, we arrive at 
the following table 
\begin{center}
\begin{tabular}{c|c|c}
 vertical arc & $z_0$ & $\kappa(z_0)$ \\
\hline
 $C_u$ & $\Bar\xi$ & $1/2$ \\
 $C_{-u}$ & $-\xi$ & $-1/2$ \\
 $C_v$ & $\chi_+(a)$ & $\kappa(a)$ \\
 $C_{-v}$ & $\chi_-(-a)$ & $\kappa(-\Bar{a})$ \\
 $C_{u-v}$ & $\chi_-(1-a)$ & $\kappa(1-\Bar{a})$ \\
 $C_{v-u}$ & $\chi_+(a-1)$ & $\kappa(a-1)$ \\
 $C_{u+v}$ & $\chi_-(a+1)$ & $\kappa(\Bar{a}+1)$ \\
 $C_{-u-v}$ & $\chi_+(-a-1)$ & $\kappa(-a-1)$ \\ 
\end{tabular}
\end{center}
where vertical arcs are labelled by the corresponding transverse cylinder and $\kappa$ is given by
$$\kappa(a) := \mathrm{Re~} \chi_+(a) = \frac{a_1}{2} - \frac{a_2}{|a|}\sqrt{1-\frac{|a|^2}{4}}.$$

In each case, $F(a,b)$ can be described as some larger rectangle 
  $[a,b]\times[c,d]$ with two or more ``corners" removed.  
These ``corners" will be described by the notation:
\begin{align*}
  NW(x,y) &:= [a,x[ ~\times~ ]y,d] & NE(x,y) &:=\; ]x,b] ~\times~ ]y,d] \\
  SW(x,y) &:= [a,x[ ~\times~ [c,y[ & SE(x,y) &:=\; ]x,b] ~\times~ [c,y[ 
\end{align*}

For $a\in\Omega^+_I$, $F(a,b)$ is given by 
  $$\left[\kappa(a-1),\frac12\right]\times[0,1] \quad\text{ minus }\quad
    SW(-1/2,1-b) \cup SE(\kappa(1-\Bar{a}),b)$$
while for $a\in\mathrm{Int~}\Omega^+_{II}$ it is given by 
  $$\left[~\kappa(a)~,\frac12\right]\times[-b,1] \quad\text{ minus }\quad
     SW(-1/2,1-b) \cup SE(~\kappa(-\Bar{a})~,0) $$
and for $a\in\Omega^+_{III}$ by 
\begin{align*}
  \left[-\frac12,\frac12\right]\times[-1,1] &\quad\text{ minus  }\quad \begin{tabular}{cc}
     $NW(~~~\kappa(a)~~,~~0~~)\;\;\cup$ & $NE(\kappa(1+\Bar{a}),b)\;\;\cup$ \\  
     $SW(\kappa(-a-1),-b)\;\;\cup$ & $SE(~\kappa(-\Bar{a})~,~0)\quad$ 
     \end{tabular} .
\end{align*}

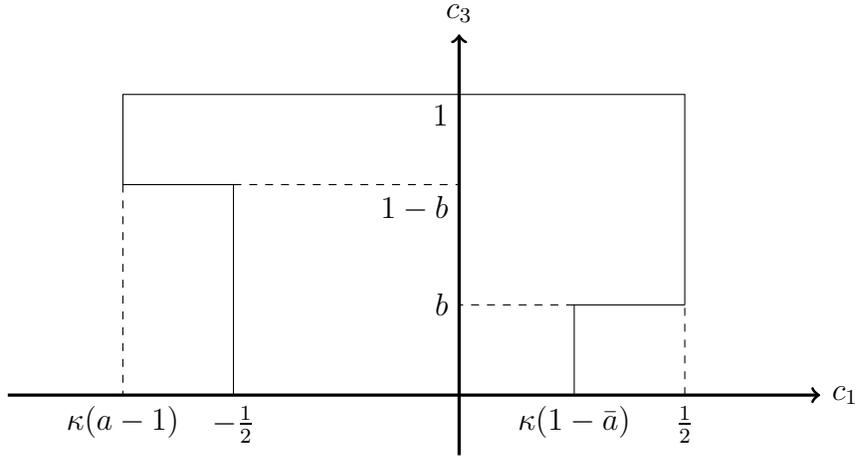
\begin{figure}
\begin{center}
\begin{tikzpicture}[xscale=6, yscale=4,
axis/.style={very thick, ->},
dashed line/.style={dashed, thin}]
% draw axes
\draw[axis] (-1,0) -- (.8,0) node (x-axis) [right] {$c_1$};
\draw[axis] (0,-.2) -- (0,1.2) node (yaxis) [above] {$c_3$};
% defining parameters
\def\kii{-.745} %values for a = .3i
\def\kiv{.255}
\def\b{0.3}
% label critical points
\draw (-.5,0) node[below] {$-\frac12$};
\draw (.5,0) node[below] {$\frac12$};
\draw (0,1) node[below left] {$1$};
\draw[dashed line] (-.5,1-\b) -- (0,1-\b) node[below left] {$1-b$};
\draw[dashed line] (\kii,1-\b) -- (\kii,0) node [below] {$\kappa(a-1)$};
\draw (\kiv,0) node[below] {$\kappa(1-\Bar{a})$} -- (\kiv,\b) -- (.5,\b);
\draw (-.5,0)--(-.5,1-\b)--(\kii,1-\b)--(\kii,1)--(.5,1)--(.5,\b);
\draw[dashed] (\kiv,\b) -- (0,\b) node[left] {$b$};
\draw[dashed] (.5,0) -- (.5,\b);
\end{tikzpicture}
\end{center}
\caption{$F(a,b)$ in the case $a=\frac{3i}{10}\in\Omega_{I}^+$ and $b=.3$}
\end{figure}

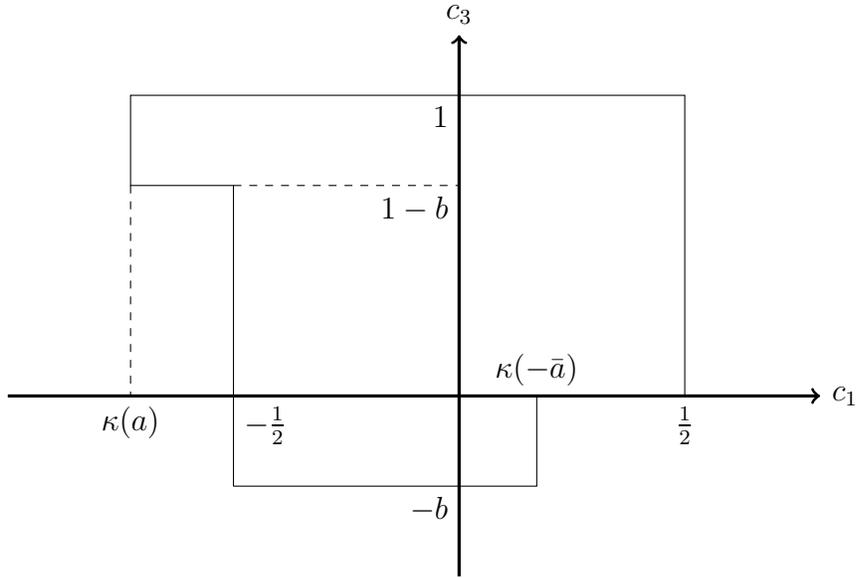
\begin{figure}
\begin{center}
\begin{tikzpicture}[xscale=6, yscale=4,
axis/.style={very thick, ->},
dashed line/.style={dashed, thin}]
% draw axes
\draw[axis] (-1,0) -- (.8,0) node (x-axis) [right] {$c_1$};
\draw[axis] (0,-.6) -- (0,1.2) node (yaxis) [above] {$c_3$};
% defining parameters
\def\kii{-.728} %values for a = -.9+.3i
\def\kiv{.172}
\def\b{.3}
% label critical points
\draw (-.5,0) node[below right] {$-\frac12$};
\draw (.5,0) node[below] {$\frac12$};
\draw (0,1) node[below left] {$1$};
\draw[dashed line] (-.5,1-\b) -- (0,1-\b) node[below left] {$1-b$};
\draw (0,-\b) node[below left] {$-b$};
\draw[dashed line] (\kii,1-\b) -- (\kii,0) node [below] {$\kappa(a)$};
\draw (\kiv,0) node[above] {$\kappa(-\Bar{a})$};
\draw (-.5,-\b)--(\kiv,-\b)--(\kiv,0)--(.5,0)--(.5,1)--(\kii,1)--(\kii,1-\b)--(-.5,1-\b)--(-.5,-\b);
\end{tikzpicture}
\end{center}
\caption{$F(a,b)$ in the case $a=\frac{-9+3i}{10}\in\Omega_{II}^+$ and $b=.3$}
\end{figure}

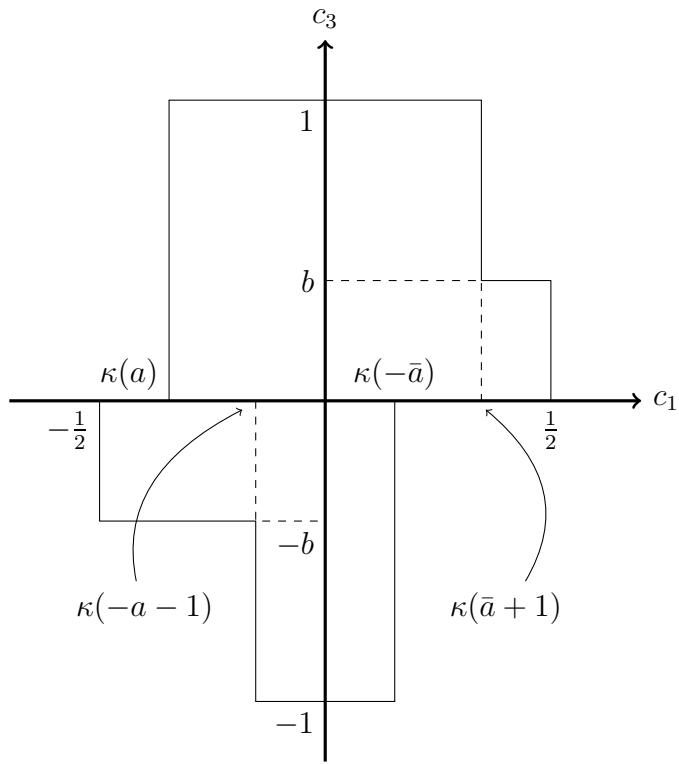
\begin{figure}
\begin{center}
\begin{tikzpicture}[xscale=6, yscale=4,
axis/.style={very thick, ->},
dashed line/.style={dashed, thin}]
% draw axes
\draw[axis] (-.7,0) -- (.7,0) node (x-axis) [right] {$c_1$};
\draw[axis] (0,-1.2) -- (0,1.2) node (yaxis) [above] {$c_3$};
% defining parameters
\def\ki{.346}
\def\kii{-.346} %values for a = -.5+.05i
\def\kiii{.-.154}
\def\kiv{.154}
\def\b{0.4}
% main outline 
\draw (\kii,0) node[above left] {$\kappa(a)$} --(\kii,1)--(\ki,1)--(\ki,\b)--(.5,\b)--(.5,0);
\draw (-.5,0)--(-.5,-\b)--(\kiii,-\b)--(\kiii,-1)--(\kiv,-1)--(\kiv,0) 
   node[above] {$\kappa(-\Bar{a})$};
% other critical points
\draw (-.5,0) node[below left] {$-\frac12$};
\draw (.5,0) node[below] {$\frac12$};
\draw (0,1) node[below left] {$1$};
\draw[dashed] (0,\b) node[left] {$b$} -- (\ki,\b) -- (\ki,0);
\draw[dashed] (\kiii,0) -- (\kiii,-\b) -- (0,-\b) node[below left] {$-b$};
\draw (0,-1) node[below left] {$-1$};
\path[->] (-.4,-.6) node[below]{$\kappa(-a-1)$} edge [bend left] (\kiii-.03,-.03);
\path[->] (.4,-.6) node[below]{$\kappa(\Bar{a}+1)$} edge [bend right] (\ki+.01,-.03);
\end{tikzpicture}
\end{center}
\caption{$F(a,b)$ in the case $a=\frac{-10+i}{20}\in\Omega_{III}^+$ and $b=.4$}
\end{figure}

It is easy to verify directly from the explicit description that $F(a,b)$ is a fundamental domain for the action of $\Z(1,b)+\Z(a_1,1)$.  From this, it follows that $\area F(a,b)=1-a_1b$.  This last claim can also be checked directly:  Indeed, in the case $a\in\Omega^+_{II}$ this boils down to the identity $\kappa(-\Bar{a}) - \kappa(a) = -a_1$, while for $a\in\Omega^+_I$ it is the same identity with $1-\Bar{a}$ instead of $a$.  The case $a\in\Omega^+_{III}$ follows by observing that the total width of NE and NW corners is $$\left(\kappa(a)+\frac12\right) + \left(\frac12 - \kappa(\Bar{a}+1) \right) = \kappa(-\Bar{a}) - \kappa(-a-1).$$

\section{Computation of $\mu_S(\Omega_7)$}
We have now established the closed form expression 
\begin{align*}
  \mu_S(\Omega_7) &= \frac{4\pi}{3} \int_{\Omega_2^+} da_1da_2 \int_0^1db \int_{F(a,b)}dc_1dc_3
    \int_{\sqrt{1-c_1^2}}^\infty\frac{dc_2}{((1-a_1b)c_2-a_2(bc_1-c_3))^3} \\
   &= \frac{2\pi}{3} \int_{\Omega_2^+} da_1da_2 \int_0^1db 
      \left(\frac{1}{1-a_1b} \int_{F(a,b)} \frac{dc_1dc_3}{\Xi^2}\right)
 \end{align*}
 where\footnote{Remark: it can be shown that $\Xi>1/24$.}  
 $$\Xi := (1-a_1b)\sqrt{1-c_1^2}-a_2(bc_1-c_3).$$  

Green's theorem applied to $\displaystyle \frac{1}{\Xi^2} = \frac{\partial Q}{\partial c_1} - \frac{\partial P}{\partial c_3}$ 
  with $Q=0$ and $\displaystyle P=\frac{1}{a_2\Xi}$ implies  
  $$\mu_S(\Omega_7) = \int_{\Omega^+_2} da_1da_2 \int_0^1db 
    \left(\frac{2\pi/3}{a_2(1-a_1b)} \int_{\partial F(a,b)} \frac{dc_1}{\Xi}\right).$$  
The substitution 
  $$c_1= \frac{2\tau}{1+\tau^2}\quad\text{ and }\quad
      dc_1= \frac{2(1-\tau^2)}{(1+\tau^2)^2} d\tau$$
and the observation 
  $$(1+\tau^2)\Xi = (1-a_1b)(1-\tau^2)-a_2(bc_1-c_3)(1+\tau^2)$$ 
readily leads to 
\begin{align*}
  \frac{dc_1}{\Xi} &= \frac{2(1-\tau^2)d\tau}{(1+\tau^2)
     (1-a_1b+a_2c_3 - 2a_2b\tau -(1-a_1b-a_2c_3)\tau^2)} \\
    &= \left(\frac{A+B\tau}{1+\tau^2} + \frac{C(a_2b+(1-a_1b-a_2c_3)\tau)+D}
     {1-a_1b+a_2c_3 - 2a_2b\tau -(1-a_1b-a_2c_3)\tau^2}\right) d\tau
\end{align*}
  where 
  $$B=C=\frac{a_2bA}{1-a_1b},\quad D=-a_2c_3A,\quad \text{ and }\quad 
    A=\frac{2(1-a_1b)}{(1-a_1b)^2+a_2^2b^2}.$$  
Since $A\tan^{-1}\tau+\frac{B}{2}\ln(1+\tau^2)$ depends on $c_1$ but not 
on $c_3$, the sum over the corners of $F(a,b)$ with alternating sign vanishes.  
The same reasoning applies to the first term of the decomposition 
  $$\frac{C}{2}\ln \left\{(1-a_1b+a_2c_3)-2a_2b\tau-(1-a_1b-a_2c_3)\tau^2\right\}
         =\frac{C}{2} \ln (1+\tau^2) + \frac{C}{2}\ln \Xi. $$ 
The second term also vanishes when summed over the corners of $F(a,b)$ with alternating sign.  To see this, we consider the pairing of the corners of $F(a,b)$ induced by the identification of vertical edges, noting that paired vertices are summed with opposite signs, and that $\Xi$ is constant on each pair.  It follows that 
  $$\mu_S(\Omega_7) = \int_{\Omega^+_2}da\int_0^1db 
     \frac{4\pi/3}{(1-a_1b)^2+a_2^2b^2}\int_{\partial F(a,b)}
     \frac{(-c_3)d\tau}{1-a_1b+a_2c_3-2a_2b\tau-(1-a_1b-a_2c_3)\tau^2}.$$  

Since we can factor $$(\phi_+-2a_2b\tau-\phi_-\tau^2)\phi_- = 
  (\sqrt{D} + a_2b + \phi_-\tau)(\sqrt{D} -a_2b -\phi_-\tau)$$  
  where $\phi_\pm=1-a_1b\pm a_2c_3$ and 
  $$D := (1-a_1b)^2 + a_2^2(b^2-c_3^2) 
      \ge |a|^2b^2-2a_1b+1-a_2^2 \ge a_2^2\left(\frac1{|a|^2}-1\right)>0$$ 
  and noting that $$\frac{d}{d\tau} \left(\ln \frac{\{\sqrt{(1-a_1b)^2+a_2(b^2-c_3^2)}+a_2b+
  (1-a_1b-a_2c_3)\tau\}^2}{1-a_1b+a_2c_3-2a_2b\tau-(1-a_1b-a_2c_3)\tau^2}\right)$$
$$ =\frac{2\sqrt{(1-a_1b)^2+a_2^2(b^2-c_3^2)}}{1-a_1b+a_2c_3-2a_2b\tau-(1-a_1b-a_2c_3)\tau^2}$$
  the rational function of $\tau$ can readily be integrated to yield
$$\mu_S(\Omega_7) =  \int_{\Omega^+_2}da\int_0^1db \sum_\nu 
     \frac{(-1)^\nu2\pi c_3/3\sqrt{D}}{(1-a_1b)^2+a_2^2b^2}
     \ln \left( \frac{\phi_+-2a_2b\tau-\phi_-\tau^2}{(\sqrt{D}+a_2b+\phi_-\tau)^2} \right)$$
  where the sum is over the corners of $F(a,b)$ labelled in such a way that the sign of, say $(1/2,b)$ in the case $a\in\Omega_I$ and $(1/2,0)$ in the other two cases, is given a plus sign.  

\textit{Remark.}  $\sqrt{D}+a_2b+\phi_-\tau>0$.  (\textit{Proof.} 
First note that $\phi_+-2a_2b\tau-\phi_-\tau^2>0$ because $\Xi>0$.  
In the case $\phi_->0$ the factors $\sqrt{D}\pm(a_2b+\phi_-\tau)$ have 
the same sign and cannot both be negative since their average is $\sqrt{D}>0$.  
In the case $\phi_-\le0$, it suffices to show that $|\phi_-|\le a_2b$ 
(since $D>0$ and $|\tau|\le1$) and this follows from 
$D=a_2^2b^2-|\phi_-|\phi_+>0$ and $|\phi_-|\le\phi_+$.)

When evaluating the integrand, it is convenient to pair up terms with the same $c_3$ that are joined by a horizontal segment on $\partial F(a,b)$.  
If $\tau_-<\tau_+$ distinguishes the endpoints, then the integrand of the triple integral 
is the expression 
$$
\frac{2\pi c_3/\sqrt{D}}{(1-a_1b)^2+a_2^2b^2} \ln \left\{
\left(\frac{\phi_+ -2a_2b\tau_+ -\phi_-\tau_+^2}
              {\phi_+ -2a_2b\tau_- -\phi_-\tau_-^2}\right)
\left(\frac{\sqrt{D}+a_2b+\phi_-\tau_-}
              {\sqrt{D}+a_2b+\phi_-\tau_+}\right)^2 \right\}
$$
summed over $c_3\in\{b,1-b,1\}$ with signs $\{+,+,-\}$ for $a\in \Omega_I^+$;
$c_3\in\{-b,1-b,1\}$ with signs $\{+,+,-\}$ for $a\in \Omega_{II}^+$; and 
$c_3\in\{-1,-b,b,1\}$ with signs $\{+,+,-,-\}$ for $a\in \Omega_{III}^+$.  

\begin{center}
\begin{tabular}{|c|c|c|c|c|}
\hline
 & $c_3$ & sign & $c_1(\tau_-)$ & $c_1(\tau_+)$ \\ \hline
I & $1$     & $-$ & $\kappa(a-1)$ & $.5$ \\ 
 & $1-b$  & $+$ & $\kappa(a-1)$ & $-.5$ \\
 &  $b$    & $+$ & $\kappa(1-\Bar{a})$ & $.5$ \\ \hline
II & $1$     & $-$ & $\kappa(a)$ & $.5$ \\ 
 & $1-b$  & $+$ & $\kappa(a)$ & $-.5$ \\
 & $-b$    & $+$ & $-.5$ & $\kappa(-\Bar{a})$ \\ \hline
III & $1$    & $-$ & $\kappa(a)$ & $\kappa(\Bar{a}+1)$ \\ 
 & $b$    & $-$ & $\kappa(\Bar{a}+1)$ & $.5$ \\
 & $-b$   & $+$ & $-.5$ & $\kappa(-a-1)$ \\
 & $-1$   & $+$ & $\kappa(-a-1)$ & $\kappa(-\Bar{a})$ \\ \hline
\end{tabular}
\end{center}
Note that $\tau$ as a function of $c_1$ is given by $\tau=\frac{1-\sqrt{1-c_1^2}}{c_1}$ apart from the removable singularity at $c_1=0$, so that 
$$\tau(a) = \frac{ 2a_1 - \sgn(a_1)|a|\sqrt{4-|a|^2}}{|a|^2+\sgn(a_1)2a_2}.$$
The expression that was fed into Octave is
\begin{align*}
3\mu_S(\Omega_7) &=  \int_{\Omega^+_2}da\int_0^1db \sum_{c_3}
     \frac{(-1)^\nu2\pi c_3}{((1-a_1b)^2+a_2^2b^2)\sqrt{D}}\ln \frac{1-x}{1+x} \\
     &=  \int_{\Omega^+_2}da\int_0^1db \sum_{c_3}
     \frac{(-1)^\nu2\pi c_3(\tau_+ -\tau_-)}{((1-a_1b)^2+a_2^2b^2)
     (\phi_+ - \tau_+\tau_-\phi_- - a_2b(\tau_+ +\tau_-))}
     \frac1x \ln \frac{1-x}{1+x}
\end{align*}     
where 
  $$x=\frac{(\tau_+ -\tau_-)\sqrt{D}} 
                           {\phi_+ \tau_+\tau_-\phi_- -a_2b(\tau_+ +\tau_-)}$$
and the value obtained by numerical integration is\footnote{In \cite{Xieu}, the factor of $3$ is missing.}
  $$3\mu_S(\Omega_7)=3.49277983865703...$$
which, using $\zeta(2)=1.2020569031...$ and $\zeta(3)=1.649340668...$, leads to $$ L_{2,1} = 1.13525697416719...$$
which is the value reported in \cite{Xieu}.

\end{document}